\documentclass{article}
\usepackage[utf8]{inputenc}
\usepackage{graphicx}
\usepackage{mathtools,amsmath,amssymb,amsfonts,amsthm}
\usepackage{mathalfa}
\usepackage{tikz,tikz-cd}
\usetikzlibrary{arrows}
\usetikzlibrary{decorations.markings}
\tikzset{degil/.style={line width=0.5pt,double distance=5pt,
        decoration={markings,
        mark= at position 0.5 with {
              \node[transform shape] (tempnode) {$\backslash\backslash$};
              }
          },
          postaction={decorate}
}
}

\tikzset{
commutative diagrams/.cd,
arrow style=tikz,
diagrams={>=open triangle 45, line width=0.5pt}}

\usetikzlibrary{calc}
\usepackage{color}
\usepackage{floatrow}   

\usepackage{url}
\usepackage{natbib}
\usepackage[backgroundcolor=white,linecolor=yellow,bordercolor=yellow,textcolor=black,textsize=small]{todonotes}
\usepackage[margin=0.5in]{geometry}

\usepackage{enumitem}

\graphicspath{{./figures/}}

\usepackage[english]{babel}

\usepackage{tabularx}

\usepackage{xcolor}

\usepackage{verbatim}
\usepackage{xspace}

\usepackage{amsfonts}
\usepackage{mathtools}
\usepackage{amsmath}

\usepackage{amssymb}

\usepackage{setspace}

\usepackage{colortbl}

\usepackage{fullpage}
\usepackage{ifthen}
\usepackage{tikz}
\usetikzlibrary{automata,positioning, arrows}
\usepackage{subcaption}
\usepackage{microtype}

\definecolor{green1}{rgb}{0.0, 0.5, 0.0}
\theoremstyle{plain}
\newtheorem{thm}{Theorem}
\newtheorem{lemma}{Lemma}

\theoremstyle{definition}
\newtheorem{fact}{Fact}
\newtheorem{defn}{Definition}

\usepackage[subpreambles=true]{standalone}
\theoremstyle{remark}
\newtheorem{myexample}{Example}
\newtheorem{rem}{Remark}
\newtheorem{assumption}{Assumption}

\newcommand{\R}{\mathbb{R}}{}
\newcommand{\N}{\mathbb{N}}

\newcommand{\Ag}[1]{\textcolor{brown}{\textbf{AG:} #1}}

\makeatletter
\newsavebox\myboxA
\newsavebox\myboxB
\newlength\mylenA

\newcommand*\pbar[1]{%
  \hbox{%
     \vbox{%
      \hrule height 0.7pt 
      \kern0.35ex
      \hbox{%
         \kern-0.0em
         \ensuremath{#1}%
         \kern-0.0em
      }%
     }%
  }%
} 

\newcommand*\xbar[2][0.75]{%
    \sbox{\myboxA}{$\m@th#2$}%
    \setbox\myboxB\null
    \ht\myboxB=\ht\myboxA%
    \dp\myboxB=\dp\myboxA%
    \wd\myboxB=#1\wd\myboxA
    \sbox\myboxB{$\m@th\pbar{\copy\myboxB}$}
    \setlength\mylenA{\the\wd\myboxA}
    \addtolength\mylenA{-\the\wd\myboxB}%
    \ifdim\wd\myboxB<\wd\myboxA%
       \rlap{\hskip 0.5\mylenA\usebox\myboxB}{\usebox\myboxA}%
    \else
        \hskip -0.5\mylenA\rlap{\usebox\myboxA}{\hskip 0.5\mylenA\usebox\myboxB}%
    \fi}
\makeatother

\title{Fixed-Time and Arbitrarily Fast Exponential Stabilization of 
Discrete-Time Switched Linear Systems}
\author{}

\date{}
\begin{document}

\maketitle

 \begin{abstract}
 In this paper, we first study the fixed-time stabilizability of discrete-time switched linear control systems. Using a geometric approach, we derive conditions under which such systems can be stabilized within a prescribed number of steps, independently of the switching sequence. We address both the mode-dependent case, where the controller has access to the active mode, and the mode-independent case, where a common feedback law must be employed. For each setting, we present constructive procedures to compute the stabilizing state-feedback gains. 
 Building on these results,
 we introduce a structural decomposition of
 switched systems, which serves to simplify stabilizability analysis and controller design.
 This allows us to establish the equivalence between fixed-time stabilizability and arbitrarily fast exponential stabilizability. The effectiveness of the proposed methods is illustrated through numerical examples.
 \end{abstract}

\section{Introduction}
In this paper, given matrices $A_1, \dots, A_M \in \mathbb{R}^{n \times n}$ and $B_1, \dots, B_M \in \mathbb{R}^{n \times m}$, we consider the discrete-time switched control system  
\begin{equation}
    x(t+1) = A_{\sigma(t)} x(t) + B_{\sigma(t)} u(t), \quad t \in \mathbb{N} \label{eq:linSS}
\end{equation}
where $\sigma: \mathbb{N} \to  \mathcal
M:= \{1, \dots, M\}$ is a switching signal and $u: \mathbb{N} \to \mathbb{R}^m$ is a control input.
The set of all possible switching sequences $\sigma = (\sigma(0), \sigma(1), \dots)$ is denoted by $\mathcal{S} := \mathcal{M}^{\infty}$. 
A key distinction in the study of system~\eqref{eq:linSS} lies in whether the switching law \(\sigma\) is treated as an external disturbance or as a controllable input to be designed in order to achieve a desired objective. 
Often, in the literature, the focus has been on the switching stabilization problem for the autonomous system ($x(t+1) = A_{\sigma(t)} x(t)$), see e.g., \cite{FiaGir16,GerCol06,WICKS1998140}.
For systems with both $\sigma$ and $u$ subject to control, see e.g., \cite{decarlo2000,HuMaLIn08,LinAnt08,ZhangAbate09}.

On the other hand, when the switching signal $\sigma$ is regarded as an unmodifiable external disturbance, several contributions have addressed the design of feedback controls $u$ that can stabilize~\eqref{eq:linSS} despite the arbitrary behavior of $\sigma$.
Significant results in this direction include~\cite{BlaMiaSav07}, where it is shown that controllers linear in each mode are, in general, not sufficient to achieve stabilization, and~\cite{DaafBer01}, which introduces poly-quadratic stability conditions for systems with time-varying uncertainty. Key advancements for stabilization are presented in~\cite{LeeDull06, LeeKha09}, where the authors develop memory-dependent (path-dependent) quadratic Lyapunov functions, leading to uniform stabilization results via LMI-based synthesis of mode-dependent feedback laws.
In this work, we focus on this latter setting, addressing the stabilization of~\eqref{eq:linSS} under arbitrary switching, where the switching signal is external and unmodifiable.
This line of research has been recently extended in~\cite{dellarossa24} through the framework of path-complete Lyapunov functions, leveraging graph-theoretic tools to derive novel LMI-based stabilization results.
In particular, this approach leads to LMI conditions that enable the synthesis of piecewise-quadratic control Lyapunov functions, from which piecewise linear feedback controllers can be systematically derived, thereby establishing a connection with earlier works on piecewise control strategies and Lyapunov-based stabilization \cite{GoeHu06,HuBla10}. 
Building on these sufficient LMI conditions, the subsequent work~\cite{lima2025} completes the framework by deriving also necessary conditions within the same tractable LMI formulation.
In other recent works, \cite{HU2024} and \cite{hu2017resilient}
investigate the problem of mode-dependent
and mode-independent
stabilization by introducing the notion of the fastest stabilizing rate, referred to here as the \emph{minimal growth rate}, which serves as a quantitative measure of a system’s stabilizability. This concept generalizes the classical joint spectral radius (JSR), originally introduced by \cite{RotaStrang1960} and thoroughly analyzed in \cite{jungers2009}, to the controlled setting. 

In the autonomous case, it is well known that a zero JSR is equivalent to fixed-time stability \cite[Proposition 2.1]{jungers2009}. It is therefore natural to ask whether the same equivalence holds in the presence of control inputs. Motivated by this, in this manuscript, we focus on characterizing the structure of switched systems \eqref{eq:linSS} that can be stabilized with arbitrarily fast convergence rates, that is, systems whose minimal growth rate is exactly zero. We prove that the aforementioned equivalence does extend to the controlled case, thus providing a natural generalization of this foundational result to systems with control inputs.


In particular, we provide a geometric characterization of fixed-time stabilizable systems, both in the mode-dependent and mode-independent settings, i.e., when either the feedback controller depends on the current-mode\footnote{In this set-up, the switching signal is still external/unmodifiable, however, observable/measurable.} or not. Moreover, we show that fixed-time stabilization can be achieved using linear feedback controllers, despite the fact that such controllers are not in general sufficient for stabilization under arbitrary switching, as previously discussed.
Additionally, we present a normal form decomposition showing that any switched system can be split into two subsystems: one that is fixed-time stable and another of reduced order that admits no non-trivial fixed-time stabilizable subspaces, while preserving the same minimal growth rate as the original system. Therefore, we provide a simpler way to analyze the \emph{minimal growth rate} of (1) through a reduced-order system, thus decreasing the overall complexity of this analysis.

The paper is organized as follows. Section 2 introduces the necessary definitions and background results on switched systems. Section 3 presents the geometric characterization of fixed-time stabilizable systems. Section 4 discusses our results on arbitrarily fast convergence. Section 5 provides a comprehensive numerical example, while Section 6 concludes the paper and outlines possible directions for future research.
\textbf{Notation:} We denote by $\N$ the set of natural numbers including $\{0\}$. Given a vector $x\in\R^n$, with $||x||$ we denote the Euclidean norm (a.k.a. $2$ norm).
 We denote by $\mathrm{Im}(A)$ the image (column space) of a matrix $A$. The preimage of a set $\mathcal{V}$ by an application $A$ is given by $A^{-1}\mathcal{V} := \{y \in \R^n \mid Ay \in \mathcal{V}  \}$. 
\section{Preliminaries}\label{sec:Prel}
We begin by introducing the main definitions and technical results that will be used throughout the paper.
\begin{defn}
System (\ref{eq:linSS}) is said to be  
\begin{enumerate}    
\item \textit{Mode-Dependent Feedback Stabilizable (MDFS)} if there exist functions $\varphi_1, \dots, \varphi_M: \mathbb{R}^n \to \mathbb{R}^m$ and constants $C > 0$ and $\rho \in [0,1)$
such that the solution of the closed-loop system
\begin{equation}
    x(t+1) = A_{\sigma(t)} x(t) + B_{\sigma(t)} \varphi_{\sigma(t)}(x(t)) \label{eq:MD-ClosedLoop}
\end{equation}
with initial condition $x(0) = x_0 \in \mathbb{R}^n$, and switching signal $\sigma\in \mathcal S$, evaluated at time $t \in \mathbb{N}$, denoted by $\Phi^D_{\sigma}(t, x_0)$, satisfies the exponential decay condition
\begin{equation}
\label{eq:Exp-MD}
    |\Phi^D_{\sigma}(t, x_0)| \leq C \rho^t |x_0|, \quad \forall  t \in \mathbb{N},\; \forall x_0 \in \R^n,\; \forall \sigma \in \mathcal{S}.
\end{equation}
\item \textit{Mode-Independent Feedback Stabilizable (MIFS)} if there exists a function $\varphi: \mathbb{R}^n \to \mathbb{R}^m$ and constants $C > 0$ and $\rho \in [0,1)$ such that the solution of the closed-loop system  
    \begin{equation}
        x(t+1) = A_{\sigma(t)} x(t) + B_{\sigma(t)} \varphi(x(t)) \label{eq:MI-ClosedLoop}
    \end{equation}
    with initial condition $x(0) = x_0 \in \mathbb{R}^n$ and switching signal $\sigma \in \mathcal S$, evaluated at time $t \in \mathbb{N}$, denoted by $\Phi^I_{\sigma}(t, x_0)$, satisfies the exponential decay condition
    \begin{equation}
\label{eq:Exp-MI}
    |\Phi^I_{\sigma}(t, x_0)| \leq C \rho^t |x_0|, \quad \forall  t \in \mathbb{N},\; \forall x_0 \in \R^n,\; \forall \sigma \in \mathcal{S}.
\end{equation}  
\end{enumerate}
\end{defn}
We now introduce fixed-time stabilizability, a stronger property requiring convergence to the origin in a fixed number of steps, uniformly over all switching sequences.
\begin{defn}
    System (\ref{eq:linSS}) is said to be 
    \begin{enumerate}    
    \item \textit{Mode-Dependent Fixed-Time Stabilizable (MDFTS)}
    if there exist $\varphi_1,\dots,\varphi_M:\mathbb{R}^n \rightarrow \mathbb{R}^m$ and a time $T \in \mathbb{N}$ such that for all initial conditions $x_0 \in \mathbb{R}^n$ and for all switching signals $\sigma \in \mathcal{S}$, the solution of the closed-loop system (\ref{eq:MD-ClosedLoop}) satisfies \begin{equation*}  
   \Phi_{\sigma}^D(t, x_0)=0, \quad \forall t \geq T.
     \end{equation*}
    
    \item \textit{Mode-Independent Fixed-Time Stabilizable (MIFTS)} if there exists a feedback map $\varphi: \mathbb{R}^n \rightarrow \mathbb{R}^m$ and a time $T \in \mathbb{N}$ such that for all initial conditions $x_0 \in \mathbb{R}^n$ and for all switching signals $\sigma \in \mathcal{S}$, the solution of the closed-loop system (\ref{eq:MI-ClosedLoop}) satisfies
   \begin{equation*}  
   \Phi_{\sigma}^I(t, x_0)=0, \quad \forall t \geq T.
     \end{equation*}
\end{enumerate}
\end{defn}
Along similar lines as in \cite{HU2024,hu2017resilient}, we define a quantitative tool to measure how fast stabilization can be achieved under feedback control.
\begin{defn}
The constant $\rho \in [0,\infty)$ is referred to as an \textit{attainable growth rate} of (\ref{eq:linSS}) 
\begin{enumerate}
    \item by a mode-dependent feedback controller if there exist functions $\varphi_1,\dots,\varphi_M: \mathbb{R}^n \rightarrow \mathbb{R}^m$ and a constant $C > 0$ such that \eqref{eq:Exp-MD} holds.

    \item by a mode-independent feedback controller if there exists a function $\varphi: \mathbb{R}^n \rightarrow \mathbb{R}^m$ and a constant $C >0$ such that \eqref{eq:Exp-MI} holds.
\end{enumerate}

The \textit{minimal growth rate} of system (\ref{eq:linSS}) under mode-independent (respectively, mode-dependent) feedback control, denoted by $\rho^I_*$ (respectively, $\rho^D_*$), is the infimum of all growth rates $\rho$ attainable with mode-independent (respectively, mode-dependent) feedback control. In particular, following the formal definitions in \cite{HU2024,hu2017resilient}, system~\eqref{eq:linSS} is mode-dependent (respectively, mode-independent) feedback stabilizable if and only if~$\rho^I_<1$ (respectively,~$\rho^D_<1$).

We say that system (\ref{eq:linSS}) admits \textit{arbitrarily fast} mode-independent (respectively, mode-dependent) stabilization if $\rho^I_* = 0$ (respectively, $\rho^D_* = 0$). In this case, exponential convergence can be achieved with an arbitrarily small rate $\rho > 0$, and we say that the system has \textit{arbitrary convergence rate}.
\end{defn}

\begin{rem} \label{rem:DynContr}
Without loss of generality, we restrict our analysis to static state-feedback controllers.  
For the class of switched linear systems considered here, \cite{lima2025} shows that static and dynamic stabilizability are equivalent.
Moreover, by \cite[Lemma~2.1]{HU2024}, a convergence rate $\rho$ is achievable if and only if the system with scaled matrices $\frac{A_j}{\rho}$, $j \in \mathcal{M}$, is stabilizable. 
This immediately implies that any rate attainable by dynamic feedback is also attainable by static feedback.  
Finally, although the feedback need not be linear in general, it can always be chosen homogeneous, see \cite{lima2025}.

\end{rem}

\section{Fixed-Time Stabilization} \label{sec:FTS}
In this section, we characterize the class of systems that are fixed-time stabilizable and provide feedback controllers that guarantee this property. For clarity, we treat the mode-dependent and mode-independent cases separately in the following subsections.
\subsection{Mode-Dependent Case}

Consider system~\eqref{eq:linSS}. To characterize the systems that are MDFTS, we introduce the recursively defined sequence \( \{E_k\} \). To start, we set $E_0 = \{0\}$. For all $k \ge 0$:
\begin{equation} 
    E_{k+1}= \bigcap_{j=1}^MA_j^{-1}(E_k+\mathrm{Im}(B_j)) \label{eq:algoMD}
\end{equation}
The sets defined by~\eqref{eq:algoMD} can be interpreted as follows: for each $k$, $E_{k+1}$ is a set of states $x$ such that, for all $j \in \mathcal{M}$ (i.e., for any active mode), there exists a control input $u_j$ satisfying $A_j x + B_j u_j \in E_k$. In other words, $E_{k+1}$ consists of states that can be driven into $E_k$ in exactly one time step, regardless of the active mode. We next highlight two key properties of this sequence.
\begin{fact} \label{fact:1}
    The sequence \( \{E_k\} \) defined in \eqref{eq:algoMD} is non-decreasing, i.e.,  $E_k \subseteq E_{k+1}$ for all $k \geq 0$.
\end{fact}
\begin{proof}
    The inclusion $E_0 \subseteq E_1$ is trivial. 
Let us now assume that for some $k \geq 1$, we have $E_{k-1} \subseteq E_{k}$. This clearly implies $E_{k-1} +\mathrm{Im}(B_j)\subseteq E_{k} +\mathrm{Im}(B_j), \text{for all} \ j \in \mathcal{M}$. Next, by applying the intersection of the pre-images by $A_j$ to both sides of this expression, we obtain
\begin{equation*}
      E_k =\bigcap_{j=1}^M A_j^{-1}\left(E_{k-1}+\mathrm{Im}(B_j)\right)\subseteq \bigcap_{j=1}^M A_j^{-1}\left(E_{k}+\mathrm{Im}(B_j)\right) = E_{k+1}
\end{equation*} 
where the equalities hold by definition of the sets $E_k$. Therefore, $E_k \subseteq E_{k+1}$ and, by induction, the statement holds for all $k \geq 0$.
\end{proof}

\begin{fact} \label{fact:2}
    The sequence \( \{ E_k \} \) defined in \eqref{eq:algoMD}  reaches a fixed point, i.e., there exists \( p \in \mathbb{N} \), with $p \leq n$, such that $E_p = E_{p+1}$. Moreover, $E_k = E_p$ for all $k \geq p$.
\end{fact}
\begin{proof}
    Since the sequence \( \{E_k\} \) is defined recursively in a finite-dimensional state space \( \mathbb{R}^n \), and each set \( E_k \) is a subspace satisfying \( E_k \subseteq E_{k+1} \) by construction, the sequence must eventually reach a fixed point. 
    In particular, there exists an integer \( p \in \mathbb{N} \), with \(p \leq n \), such that \( E_p = E_{p+1} \).
    Moreover, this equality implies that the sequence becomes constant from step \( p \) onward, i.e., \( E_k = E_p \) for all \( k \geq p \). 
\end{proof}

In what follows, we define some matrices and vectors associated with the sequence $\{E_k\}$ which will be used in the subsequent statements. Let us start by defining $Q_1=\begin{bmatrix}
    q_1&\dots&q_{e_1}
\end{bmatrix}$ as a matrix whose columns form a basis of $E_1$. 
For all $1 \leq i \leq e_{1}$, for all $j \in \mathcal{M}$, there exists $u_{i,j}$ such that $A_jq_i + B_j u_{i,j} \in E_0$.
Let us now define $U_{1,j}=\begin{bmatrix}
    u_{1,j},\dots,u_{e_1,j}
\end{bmatrix}$, the matrix of control inputs associated to $Q_1$. 
Then, let $p \in \mathbb{N}$ be such that $E_p$ is the fixed point of the sequence $ \{E_k\}$. For all $1 \leq k \leq p-1$, we construct the basis matrix $Q_{k+1}$ of $E_{k+1}$ by extending the basis $Q_k=\begin{bmatrix}
    q_1 & \dots & q_{e_k}
\end{bmatrix}$ of $E_k$ with a completion $Q_{k+1}'=\begin{bmatrix}
    q_{e_k+1} & \dots & q_{e_{k+1}}
\end{bmatrix}$ as follows: $Q_{k+1}= \begin{bmatrix}
    Q_k & Q_{k+1}'
\end{bmatrix}$. 
For all $e_k+1 \leq i \leq e_{k+1}$, for all $j \in \mathcal{M}$, there exists $u_{i,j}$ such that $A_jq_i + B_j u_{i,j} \in E_k$. Thus, we can define the matrices of control inputs corresponding to $Q_{k+1}$ as $U_{k+1,j}=\begin{bmatrix}
   U_{k,j} & U_{k+1,j}'
\end{bmatrix}$, with $U_{k+1,j}' = \begin{bmatrix}
     u_{e_{k}+1,j} & \dots & u_{e_{k+1},j}
\end{bmatrix}$. 
To prepare for the next result, we first illustrate how the gain of the mode-dependent linear feedback controller can be constructed from the matrices defined above.
For all \(0 \leq k \leq p-1\) and for all \(j \in \mathcal{M}\), it holds that
\begin{equation*}
A_jQ_{k+1}' + B_jU_{k+1,j}' = Q_k Y_{k+1,j},
\end{equation*}
which is equivalent to
\begin{equation} 
A_jQ_{k+1}' = \begin{bmatrix}-B_j & Q_k\end{bmatrix} 
\begin{bmatrix} U_{k+1,j}' \\ Y_{k+1,j} \end{bmatrix}, 
\label{eq:lin_system}
\end{equation}
where \(Y_{k+1,j}\) is an auxiliary matrix.
A solution to system~\eqref{eq:lin_system} exists by construction. 
If multiple solutions exist, one can, for example, choose the one with minimum norm.

Next, consider \(Q_p\) and \(U_{p,j}\). We define the feedback gains \(K_j\) by
\[
K_j Q_p = U_{p,j},
\]
where a solution is given by
\[
K_j = U_{p,j} Q_p^+, \quad \text{with } Q_p^+ Q_p = I_{e_p},
\]
which explicitly yields
\[
K_j = U_{p,j} (Q_p^\top Q_p)^{-1} Q_p^\top.
\]
In the special case where \(E_p = \mathbb{R}^n\) (i.e., \(\mathrm{rank}(Q_p) = n\)), this reduces to \(K_j = U_{p,j} Q_p^{-1}\).

The next result establishes that membership of a state in the subspace \(E_k\) is equivalent to the capability of steering it to the origin in \(k\) steps using a mode-dependent controller, regardless of the switching sequence.

\begin{lemma} \label{lemma:LemmaMD}
Let \{$E_k$\} be the sequence defined in~\eqref{eq:algoMD} and let $p \in \mathbb{N}$ be such that $E_p=E_{p+1}$. Let $Q_p$ and $U_{p,j}$ be defined as above. The following statements are equivalent, for all $k=0,\dots,p-1$:
\begin{enumerate}[label=(\roman*)]
    \item $x_0 \in E_k$,
    \item There exist feedback functions \( \varphi_1, \dots, \varphi_M: \mathbb{R}^n \to \mathbb{R}^m \) such that for all \( \sigma \in \mathcal{S} \), \( \Phi_{\sigma}^D(k,x_0) = 0 \),
        \item There exists a mode-dependent linear feedback law of the form \( \varphi_j(x) = K_j x \), for all \( j \in \mathcal{M} \), such that for all \( \sigma \in \mathcal{S} \), \( \Phi_{\sigma}^D(k,x_0) = 0 \). 
        
        Moreover, we can choose \( K_j = U_{p,j} (Q_p^{\top} Q_p)^{-1} Q_p^{\top} \), and if $E_p=\mathbb{R}^n$, $K_j=U_{p,j}Q_p^{-1}$.
\end{enumerate}
\end{lemma}
\begin{proof}

$(ii) \Rightarrow (i)$. $\Phi_{\sigma}^D(0,x_0)=0$ implies trivially that $x_0 \in E_0$. Assume that if there exist $\varphi_1^k,\dots,\varphi_M^k$ such that for all $\sigma \in \mathcal{S}, \Phi_{\sigma}^D(k,x_0)=0$, then $x_0 \in E_k$. If there exist $\varphi_1^{k+1},\dots,\varphi_M^{k+1}$ such that for all $\sigma \in \mathcal{S}, \Phi_{\sigma}^D(k+1,x_0)=0$, it means that, for all $j \in \mathcal{M}$, there exists $u_j \in \mathbb{R}^m$ such that $A_jx_0+B_ju_j \in E_k$, which is equivalent to $x_0 \in \bigcap_{j=1}^M A_j^{-1}(E_k+\mathrm{Im}(B_j))$, that is $E_{k+1}$ by definition. 

$(i) \Rightarrow (iii)$. 
Consider \(q_i\) and the associated \(u_{i,j}\), for \(1 \leq i \leq e_{k+1}\), as defined above Lemma~\ref{lemma:LemmaMD}. By definition, for each \(q_i \in E_{k+1}\) and every \(j \in \mathcal{M}\), $A_j q_i + B_j u_{i,j} \in E_k$.
Since the vectors \(q_i\) are linearly independent, for each $j \in \mathcal{M}$, we can define \(K_j\) so that \(K_j q_i = u_{i,j}\) for all \(i=1,\dots,e_{k+1}\). By construction, the matrices \(Q_{k+1}'\) and \(U_{k+1,j}'\) are subblocks of \(Q_p\) and \(U_{p,j}\), respectively. Hence the explicit choice $K_j = U_{p,j}(Q_p^\top Q_p)^{-1}Q_p^\top$ satisfies \(K_j Q_{k+1}' = U_{k+1,j}'\) for every \(k=0,\dots,p-1\). Consequently, for all \(x_0 \in E_{k+1}\) and all \(j \in \mathcal{M}\), $(A_j + B_j K_j)x_0 \in E_k.$
Since each feedback matrix maps 
\(E_{k+1}\) into \(E_k\), iterating this property ensures that, for any 
\(x_0 \in E_{k+1}\) and any switching sequence \(\sigma \in \mathcal{S}\), 
the state reaches the origin in at most \(k\) steps, i.e., 
\(\Phi_\sigma^D(k, x_0) = 0\).

$(iii) \Rightarrow (ii)$ follows directly, since (ii) accounts for general nonlinear feedback, while (iii) represents the specific case where the feedback is linear.
\end{proof}
Building on the previous developments, we present the main result of this section. Clearly, if the fixed point \( E_p \) of the sequence \(\{E_k\}\) coincides with the entire state space \(\mathbb{R}^n\), then the feedback controller constructed in Lemma~\ref{lemma:LemmaMD} steers any initial condition to the origin in at most \(p\) steps, regardless of the switching sequence. The following theorem formalizes this equivalence.

\begin{thm} \label{thm:MDFTS}
    Let $p \in \mathbb{N}$, and let $E_p$ be the fixed point of the sequence defined in (\ref{eq:algoMD}).
    System \eqref{eq:linSS} is MDFTS if and only if $E_p = \mathbb{R}^n$.\\
 Moreover, a mode-dependent feedback controller achieving Mode-Dependent Fixed-Time Stabilizability can be chosen linear, and is given by $u(t)=K_{\sigma(t)}x(t)$, with $K_j=U_{p,j}Q_p^{-1}, j \in \mathcal{M}$.
\end{thm}
\begin{proof}
    \textbf{Sufficiency}. Assume \( E_p = \mathbb{R}^n \). Consider the control law \( u(t) = K_{\sigma(t)}x(t) \). Defining $K_j$, for all $j \in \mathcal{M}$, as shown in the proof of Lemma \ref{lemma:LemmaMD}, it follows that for all initial conditions $x_0 \in E_p$ we reach the origin in $p$ time steps . Since $E_p=\mathbb{R}^n$ the result follows.
    
 \textbf{Necessity}. Assuming that the system is MDFTS with $T=p$, it follows straightfowardly from Lemma \ref{lemma:LemmaMD} that $E_p=\mathbb{R}^n$.  
\end{proof}

\subsection{Mode-Independent case}

Consider the discrete-time switched linear control system \eqref{eq:linSS}. To characterize fixed-time stabilizability in the mode-independent case, we first recursively define a sequence of sets $\{E_k\}$, similarly to the mode-dependent case. Set $E_0=\{0\}$. For all $k \geq 0$:
\begin{equation} \label{eq:algoMI}
 E_{k+1}=\bar A^{-1}(E_k^M+\mathrm{Im(\bar B})),\end{equation} 
with 
$$\bar A=\begin{pmatrix}
    A_1\\
    \vdots \\ 
    A_M
\end{pmatrix} \text{ and } \bar B=\begin{pmatrix}
    B_1 \\ 
    \vdots \\ 
    B_M
\end{pmatrix}$$
The expression above follows the same idea as in the mode-dependent case, and it translates to the existence of a control input such that for all modes, a state belonging to $E_{k+1}$ is driven to $E_k$ in one step, that is, given $x \in E_{k+1}$, there exists $u,  \text{such that for all} \ j \in \mathcal{M}, A_jx + B_ju \in E_k$.

We now present the analogues of Facts \ref{fact:1} and \ref{fact:2} for the mode-independent case. Their proofs are omitted, as they follow the same lines as the proofs of Facts \ref{fact:1} and \ref{fact:2}.
\begin{fact} 
    The sequence \( \{E_k\} \) defined in (\ref{eq:algoMI}) is non-decreasing, i.e.,  $    E_k \subseteq E_{k+1}$ for all $k \geq 0$.

\end{fact}

\begin{fact}
    The sequence \( \{ E_k \} \) defined in \eqref{eq:algoMI}  reaches a fixed point, i.e., there exists \( p \in \mathbb{N} \), with $p \leq n$, such that $E_p = E_{p+1}$. Moreover, $E_k = E_p$ for all $k \geq p$.
    
\end{fact}

Now, let $p \in \mathbb{N}$ be such that $E_p$ is the fixed point of the sequence $\{E_k\}$. For all $0 \leq p-1$, we can define the matrices $Q_{k+1}$ and $U_{k+1}$ as in the previous section. Set $Q_1=\begin{bmatrix}
    q_1&\dots&q_{e_1}
\end{bmatrix}$, with $q_1,\dots,q_{e_1}$ a basis of $E_1$. For all $1 \leq i \leq e_1$, there exists $u_i$ such that for all $j \in \mathcal{M}$, $A_jq_i+B_ju_i \in E_0$. Set $U_1=\begin{bmatrix}
    u_{1},\dots,u_{e_1}
\end{bmatrix}$, the matrix of control inputs associated to $Q_1$. Then, for all $1 \leq k \leq p-1$, the basis matrix of $E_{k+1}$ is $Q_{k+1}=\begin{bmatrix}
    Q_k & Q_{k+1}'
\end{bmatrix}$, with $Q_{k+1}'=\begin{bmatrix}
    q_{e_k+1} & \dots & q_{e_{k+1}}
\end{bmatrix}$. For all $e_k+1 \leq i \leq e_{k+1}$, there exists $u_i$ such that for all $j \in \mathcal{M}$, $A_jq_i+B_ju_i \in E_k$ thus, finally, we can define $U_{k+1}=\begin{bmatrix}
   U_{k} & U_{k+1}'
\end{bmatrix}$, with $U_{k+1}' = \begin{bmatrix}
     u_{e_{k}+1} & \dots & u_{e_{k+1}}
\end{bmatrix}$.

Moreover, for all $0 \leq k \leq p-1$, for all $j \in \mathcal{M}$,
\begin{equation*}
    \bar A Q_{k+1}' + \bar B U_{k+1}'= \operatorname{diag}_M(Q_k)
 Y_{k+1}
\end{equation*}
or, 
\begin{equation} \label{eq:lin_systemMI}
    \bar A Q_{k+1}'= 
    \begin{pmatrix}
        -\bar B & \operatorname{diag}_M(Q_k)
    \end{pmatrix} 
    \begin{pmatrix}
        U_{k+1}'\\ Y_{k+1}
    \end{pmatrix},
\end{equation} 
where $\operatorname{diag}_M(Q_k)$ denotes the block-diagonal matrix with $M$ copies of $Q_k$ on its diagonal, and $Y_{k+1}$ is an auxiliary matrix.  
A solution to system~\eqref{eq:lin_systemMI} exists by construction; if it is not unique, one can, for example, choose the solution with minimum norm.
 
Next, considering $Q_p$ and $U_p$, we define the feedback gain $K$ as
\[
K Q_p = U_p.
\]
A solution is then given by
\[
K = U_p Q_p^+, \quad \text{with } Q_p^+ Q_p = I_{e_p},
\]
which can be explicitly written as
\[
K = U_p (Q_p^\top Q_p)^{-1} Q_p^\top.
\]
In the special case where $E_p = \mathbb{R}^n$ (i.e., $\mathrm{rank}(Q_p) = n$), this reduces to
\[
K = U_p Q_p^{-1}.
\]
The next lemma establishes an analogous equivalence as in Lemma \ref{lemma:LemmaMD}, for the mode-independent case.
\begin{lemma} \label{LemmaMIequivalences}
    Let \{$E_k$\} be the sequence defined in (\ref{eq:algoMI}) and let $p \in \mathbb{N}$ be such that $E_p=E_{p+1}$. Let $Q_p$ and $U_p$ be defined as above. The following statements are equivalent, for all $k=0,\dots,p-1$: 
\begin{enumerate}[label=(\roman*)]
    \item $x_0 \in E_k$,
    \item There exists a feedback function \( \varphi: \mathbb{R}^n \to \mathbb{R}^m \) such that for all \( \sigma \in \mathcal{S} \), \( \Phi_{\sigma}^I(k,x_0) = 0 \),
        \item There exists a mode-independent linear feedback law of the form \( \varphi(x) = K x \), such that for all \( \sigma \in \mathcal{S} \), \( \Phi_{\sigma}^I(k,x_0) = 0 \). 
        
        Moreover, we can choose \( K = U_{p} (Q_p^{\top} Q_p)^{-1} Q_p^{\top} \), and if $E_p=\mathbb{R}^n$, $K=U_{p}Q_p^{-1}$.
\end{enumerate}
\end{lemma}

Having established the previous lemma, we now state the main result for the mode-independent case.
\begin{thm} \label{thm:MIFTS}
    Let $p \in \mathbb{N}$, and let $E_p$ be the fixed point of the sequence defined in \eqref{eq:algoMI}.
    System \eqref{eq:linSS} is MIFTS if and only if $E_p = \mathbb{R}^n$.\\
 Moreover, the mode-independent feedback controller achieving MIFTS is linear and time-invariant, and is given by $u(t)=Kx(t)$, with $K=U_pQ_p^{-1}$.
\end{thm}

\section{Arbitrary Convergence Rate}\label{sec:ACR}
In this section, we investigate arbitrarily fast stabilization and show that it is is equivalent to fixed-time stabilization, for which a structural characterization was given in the previous section. As a consequence, we also obtain a characterization of systems that admit arbitrarily fast convergence.
We start with the following normal form decomposition result. 

\begin{thm} \label{thm:normalform}
    Consider system~(\ref{eq:linSS}) and assume that its minimal growth rate under mode-independent (respectively, mode-dependent) feedback controller is $ \rho_*^I $ ($ \rho_*^D$).  
Then, there exist a change of feedback and a change of coordinates under which the system admits the representation
\begin{equation}
    z(t+1)= \begin{pmatrix}
        A_{\sigma(t)}^{yy} & A_{\sigma(t)}^{y \xi} \\
        \mathbf{0} & A_{\sigma(t)}^{\xi \xi}
    \end{pmatrix}z(t)+\begin{pmatrix}
        B_{\sigma(t)}^y \\ B_{\sigma(t)}^{\xi}
    \end{pmatrix}v(t),
\end{equation}
where \( z = \begin{pmatrix} y \\ \xi \end{pmatrix} \), and the following hold:
\begin{enumerate}
    \item The autonomous \( y \)-subsystem, described by \( y(t+1) = A_{\sigma(t)}^{yy} y(t) \), is fixed-time stable, i.e., there exists \( T \in \mathbb{N} \) such that for all $y(0) \in \mathbb{R}^{n_p}$, for all $\sigma \in \mathcal{S}$, \( y(t) = 0 \), for all \( t \geq T \).
   
    \item The \( \xi \)-subsystem, described by \( \xi(t+1) = A_{\sigma(t)}^{\xi \xi} \xi(t) + B_{\sigma(t)}^{\xi} v(t) \), admits no non-trivial fixed-time stabilizable subspaces. Furthermore, its minimal growth rate, denoted by \( \rho_*^{I,\xi} \) (respectively, \( \rho_*^{D,\xi} \)), is equal to the minimal growth rate of system~(\ref{eq:linSS}), i.e., \( \rho_*^{I,\xi} = \rho_*^I \) (\( \rho_*^{D,\xi} = \rho_*^D \)).
\end{enumerate}

\end{thm}
\begin{proof}
    The proof is presented only for the mode-independent case, as the mode-dependent case can be treated analogously.
    Consider system~(\ref{eq:linSS}). Let \( E_p \subseteq \mathbb{R}^n \) be the fixed point of the sequence \( \{E_k\} \) defined in (\ref{eq:algoMI}), i.e., the largest set of initial conditions for which the system is fixed-time stabilizable in at most \( p \) steps using the control gains \( K \) computed via Lemma~\ref{lemma:LemmaMD}. Let the control input be $u(t) = K x(t) + v(t).$
    Thus, the system dynamics can be written as
    \begin{equation*}
        x(t+1) = (A_{\sigma(t)} + B_{\sigma(t)} K) x(t) + B_{\sigma(t)} v(t) = \bar{A}_{\sigma(t)} x(t) + B_{\sigma(t)} v(t)
    \end{equation*}
    
    Next, define the coordinate transformation \( z = P x = \begin{pmatrix} y \\ \xi \end{pmatrix} \), with \( y \in E_p \). The matrix \(P\) is chosen so that its first \(n_p\) columns form a basis of \(E_p\), and the remaining columns complete a basis of \(\mathbb{R}^n\).
In these coordinates, the system has the following block-triangular form: 
    \begin{equation} \label{eq:NormalFormMI}
        z(t+1) = \begin{pmatrix}
            A_{\sigma(t)}^{yy} & A_{\sigma(t)}^{y \xi} \\
            \boldsymbol{0} & A_{\sigma(t)}^{\xi \xi}
        \end{pmatrix} z(t) + \begin{pmatrix}
            B_{\sigma(t)}^y \\ 
            B_{\sigma(t)}^{\xi}
        \end{pmatrix} v(t)
    \end{equation}

Notice that the minimal growth rate of system \eqref{eq:NormalFormMI} coincides with that of system \eqref{eq:linSS}, as it is preserved under changes of feedback and coordinates.
Moreover, the autonomous \(y\)-subsystem is defined on the invariant subspace \(E_p\). 
Under this change of feedback and because of Theorem \ref{thm:MIFTS} , all trajectories starting in \(E_p\) reach the origin in at most \(p\) steps, uniformly over all admissible switching signals \(\sigma \in \mathcal{S}\). 
Hence, the autonomous \(y\)-subsystem is fixed-time stable.

Now, we prove that $\rho_*^{I,\xi} \leq \rho_*^I$. Let $\rho > \rho_*^I$. Then there exists a static feedback controller $\psi(z(t))$ such that for all $z(0) \in \mathbb{R}^n$ and all $\sigma \in \mathcal{S}$ \\
\[
\|z(t)\| \leq C \rho^t \|z(0)\|.
\]

Now, let us consider an initial state $z(0)=\begin{pmatrix} 0 \\ \xi(0) \end{pmatrix}$. 
From \eqref{eq:NormalFormMI}, $y(t)$ can be written under the form:
\[
y(t) = \Psi(\xi(t-1), \dots, \xi(0), \sigma(t-1), \dots, \sigma(0)).
\]
Thus, the evolution of \( y(t) \) depends only on the trajectory of \( \xi \) and the switching history. Consequently, \( v(t) = \psi(z(t)) \) can be equivalently expressed as
\[
v(t) = \bar{\psi}(\xi(t), \dots, \xi(0), \sigma(t-1), \dots, \sigma(0)),
\]
which shows that \( v(t) \) is independent of the current mode value $\sigma(t)$ and can be implemented via a feedback law with memory.

The \( \xi \)-subsystem evolves independently of \( y \) and is governed by

\[
\xi(t+1) = A_{\sigma(t)}^{\xi\xi} \xi(t) + B_{\sigma(t)}^{\xi} v(t).
\]
Hence, we can write
\[
\|\xi(t)\| \leq \|z(t)\| \leq C \rho^t \|z(0)\| = C \rho^t \left\|\begin{pmatrix} y(0) \\ \xi(0) \end{pmatrix}\right\| = C \rho^t \left\|\begin{pmatrix} 0 \\ \xi(0) \end{pmatrix}\right\|= C \rho^t \|\xi(0)\|.
\]
which shows that $\rho$ is an attainable growth rate for the $\xi$-subsystem using the current mode-independent feedback with memory. By Remark \ref{rem:DynContr}, the same growth rate can also be achieved by a static, memoryless mode-independent feedback controller. 
We can then conclude
\[
\rho_*^{I,\xi} \leq \rho_*^I.
\]

Now, let us prove that $\rho_*^{I,\xi} \geq \rho_*^I$.\\
Let $\rho > \rho_*^{I,\xi}$. Then there exists a static feedback $\varphi(\xi(t))$ such that
\[
\|\xi(t)\| \leq C_{\xi} \rho^t \| \xi(0)\|
\]
Next, consider the $y$-subsystem
\begin{equation*}
    y(t+1)=A_{\sigma(t)}^{yy}y(t) + A_{\sigma(t)}^{y\xi}\xi(t) + B_{\sigma(t)}^{y}\varphi(\xi(t)).
\end{equation*}
and define $d(t)= A_{\sigma(t)}^{y\xi}\xi(t) + B_{\sigma(t)}^{y}\varphi(\xi(t))$. By Remark \ref{rem:DynContr}, we can assume without loss of generality that $\varphi(\xi(t))$ is homogeneous of degree 1. It follows that
\begin{equation*}
    ||d(t)|| \leq C || \xi(t)|| \leq C_d \rho^t || \xi(0)||,
\end{equation*}
where $C_d=C \cdot C_{\xi}$.
We denote by \( \phi^y_{\sigma}(t, s) \) the state transition matrix associated with the \( y \)-subsystem, defined as
\[
\phi^y_{\sigma}(t, s) := \prod_{i = s}^{t-1} A^{yy}_{\sigma(i)} = A^{yy}_{\sigma(t-1)} \cdots A^{yy}_{\sigma(s)}.
\]

The solution of the $y$-subsystem at time $t$, starting from $y(0)$, is given by
\begin{equation*}
    y(t) = \phi^y_{\sigma}(t,0)y(0) + \sum_{s=0}^{t-1}\phi^y_{\sigma}(t,s+1) d(s).
\end{equation*}
By construction of the fixed-time stable $y$-subsystem, we have $\phi_\sigma^y(t,0) = 0$ for all $t \ge p$. 
Moreover, for the sum, the contribution of $d(s)$ at time $s$ affects $y(t)$ only for at most $p$ steps due to the FTS property. 
Hence, all terms in the sum with $s < t-p$ vanish.
 Therefore, for all $t \geq p$, we have $$y(t)= \sum_{s=t-p}^{t-1}\phi^y_{\sigma}(t,s+1) d(s),$$and there exists a constant $\beta$ such that $||A_{\sigma}^{yy}|| \leq \beta$ for all modes.
Hence, we can write
\begin{equation*}
    ||y(t)|| \leq \sum_{s=t-p}^{t-1} ||\phi_{\sigma}^y(t,s+1)|| \cdot ||d(s)|| \leq \sum_{s=t-p}^{t-1} \beta^{t-s-1} C_d \rho^{t+(s-t)}||\xi(0)|| \leq \end{equation*} \begin{equation*} \leq\sum_{s=1}^{p} (\frac{\beta}{\rho})^{s} \beta^{-1} C_d \rho^{t}||\xi(0)|| = \bar \beta \rho^t ||\xi(0)||.
\end{equation*}
For $t < p$, a similar argument gives
\begin{equation*}
    ||y(t)|| \leq \beta^t |y(0)|| + \bar \beta \rho^t || \xi(0)|| \leq (\frac{\beta}{\rho})^t \rho^t ||y(0)|| + \bar \beta \rho^t || \xi(0)||
\end{equation*}
Define $\gamma= \displaystyle \max_{t=\{0,\dots,p\}} (\frac{\beta}{\rho})^t$. 
Then, for all $t \ge 0$ we have \begin{equation*}
    ||y(t)|| \leq \gamma \rho^t ||y(0)|| + \bar \beta \rho^t || \xi(0)||
\end{equation*}
Since also $|| \xi(t)|| \leq C_{\xi} \rho^t ||\xi(0)||$, combining the two terms yields
\begin{equation*}
    ||z(t)|| = \left\| \begin{pmatrix}
        y(t) \\ \xi(t)
    \end{pmatrix} \right\| \leq \gamma \rho^t||y(0)|| + (\bar \beta + C_{\xi})\rho^t||\xi(0)|| \leq C \rho^t ||z(0)||,
\end{equation*}
with $C := \max\{ \gamma, \bar{\beta} + C_\xi \}.
$
Finally, we conclude that $\rho_*^{I,\xi} \geq \rho_*^I$. \\
Since we have also shown that \( \rho_*^{I,\xi} \leq \rho_*^I \), it follows that $$\rho_*^{I,\xi} = \rho_*^I.$$
\end{proof}
We showed that any switched system can be transformed so as to isolate its fixed-time stable subsystem, thereby reducing the analysis to a lower-dimensional subsystem. This decomposition will be instrumental in proving the following theorem, which represents a central result of this work.
Before proceeding, we state a standing assumption on the input matrices. 
\begin{assumption} \label{ass:fullrankB}
The input matrices satisfy
$\operatorname{rank}(\bar B) = m,$
where $\bar B$ is defined in (8).
\begin{rem}
This assumption is not restrictive. 
If the  matrix $\bar B$ has rank $\tilde m < m$, 
one can apply an input transformation $ \tilde u = T^{-1}u$ with a suitable matrix $T \in \mathbb{R}^{m \times \tilde m}$ 
so that the transformed matrices $\tilde B_j = B_j T$ have full column rank $\tilde m$. 
The redundant input directions can thus be eliminated without loss of generality.
\end{rem}
\end{assumption}

\begin{thm} \label{thm:MIConj}
    System (\ref{eq:linSS}) admits arbitrarily fast mode-independent (respectively, mode-dependent) stabilization, i.e., $\rho^I_* = 0$ ($\rho^D_* = 0$), if and only if it is MIFTS (MDFTS).
    
\end{thm}
The proof of Theorem~\ref{thm:MIConj} relies on the following lemmas.
\begin{lemma} \label{lemma:MIConj} 
    If $\rho_*^I = 0$, then there exist $x \neq 0$ and $u \in \mathbb{R}^m$ such that $A_jx + B_ju = 0$ for all $j \in \mathcal{M}$.
\end{lemma}

\begin{proof}
We proceed by contradiction. Assume that for every $x \neq 0$ and every $u \in \mathbb{R}^m$, there exists at least one $j \in \mathcal{M}$ such that 
\[
A_jx + B_j u \neq 0.
\]
Under Assumption 1, $f(x,u) := \max_{j \in \mathcal{M}} \|A_j x + B_j u\|$ is unbounded over $\mathbb{R}^m$.
 By continuity of $f$, the minimum is therefore attained at some finite $\bar u \in \mathbb{R}^m$.
 
Define
\[
w(x) := \min_{u \in \mathbb{R}^m} \max_{j \in \mathcal{M}} ||A_jx+B_ju||.
\]
To show that $w$ is Lipschitz continuous, let $x, x' \in \mathbb{R}^n$ be arbitrary. 
Let $\bar u_x$ and $\bar u_{x'}$ denote minimizers of $w(x)$ and $w(x')$, respectively, i.e.,
\[
w(x) = \max_{j \in \mathcal{M}} \|A_j x + B_j \bar u_x\|, 
\qquad
w(x') = \max_{j \in \mathcal{M}} \|A_j x' + B_j \bar u_{x'}\|.
\]
First, using the minimizer $\bar u_{x'}$ of $w(x')$, we have
\[
w(x) \le \max_{j \in \mathcal{M}} \|A_j x + B_j \bar u_{x'}\|.
\]
By the triangle inequality,
\[
\max_{j \in \mathcal{M}} \|A_j x + B_j \bar u_{x'}\| \leq \max_{j \in \mathcal{M}} \|A_j x' + B_j \bar u_{x'}\| 
+ \max_{j \in \mathcal{M}} \|A_j (x - x')\|,
\]
which gives
\[
w(x) \leq w(x') + \max_{j \in \mathcal{M}} \|A_j (x - x')\|.
\]
Similarly, swapping the roles of $x$ and $x'$,
\[
w(x') \leq w(x) + \max_{j \in \mathcal{M}} \|A_j (x' - x)\|.
\]
Combining the two inequalities, we obtain
\[
||w(x') - w(x)|| \le \max_{j \in \mathcal{M}} \|A_j (x' - x)\| = L \|x' - x\|,
\]
where $L := \max_{j \in \mathcal{M}} \|A_j\|$.
Hence, $w$ is Lipschitz continuous.
Consequently, when we restrict $w$ to the unit sphere $\{x \in \mathbb{R}^n : \|x\| = 1\}$, which is compact, the continuous function $w$ attains its minimum. 
By assumption, there is no common input $u$ that simultaneously cancels $A_jx+B_ju$ for all modes.
By the previous argument, $w(x) > 0$ for all $x \neq 0$, so there exists $\alpha > 0$ such that
\[
w(x) \ge \alpha, \qquad \forall x \in \mathbb{R}^n, ||x||=1.
\]
Finally, for any trajectory starting from $x(0) \neq 0$, the homogeneity of the system implies
\[
\|x(t)\| \ge \alpha^t \|x(0)\|, \quad \forall t,
\]
which shows that $\rho_*^I \ge \alpha > 0$, contradicting the assumption $\rho_*^I = 0$.
\end{proof}

We now state the corresponding result for the mode-dependent case.  

\begin{lemma}
    If $\rho_*^D = 0$, then there exists $x \neq 0$ such that for all $j \in \mathcal{M}$ there exists $u_j \in \mathbb{R}^m$ with $A_j x + B_j u_j = 0$.
\end{lemma}
\begin{proof}
We proceed by contradiction, in analogy with Lemma~\ref{lemma:MIConj}.  
Assume that for every $x \neq 0$, there exists at least one $j \in \mathcal{M}$ such that $\forall u_j \in \mathbb{R}^m$
\[
A_j x + B_j u_j \neq 0 
\]

Consider any \(x \neq 0\), and let \(j_x \in \mathcal{M}\) be a mode such that for all \(u \in \mathbb{R}^m\),
\[
A_{j_x} x + B_{j_x} u \neq 0.
\]
 Define
\[
w(x) := \max_{j \in \mathcal{M}} \min_{u_j \in \mathbb{R}^m} \|A_j x + B_j u_j\|.
\]
Therefore, 
\[
w(x) \ge \min_{u \in \mathbb{R}^m} \|A_{j_x} x + B_{j_x} u\| > 0.
\]

The rest of the argument follows exactly as in the mode-independent case: $w$ is Lipschitz continuous, attains a strictly positive minimum on the unit sphere, and homogeneity implies that any trajectory grows at at a strictly positive rate. This contradicts $\rho_*^D = 0$.
\end{proof}

We are now in position to prove Theorem \ref{thm:MIConj}. 
We provide the proof in the mode-independent setting; the mode-dependent case follows from similar arguments.
\begin{proof}
Necessity. Consider system \eqref{eq:linSS} and the sequence $\{E_k\}$ constructed in \eqref{eq:algoMI}. Applying the feedback law $\varphi(t)=Kx(t)$ we have that $\Phi_{\sigma}(t,x_0)=0$, for all $t \geq p$, which trivially implies that for all $\rho > 0$, there exists $C>0$ such that 
\begin{equation} \label{acrProofLemmaMI} 
   | \Phi_{\sigma}(t,x_0)| \leq C \rho^t |x_0|,
\end{equation}
implying that $\rho^I_*=0$. Now we will prove that (\ref{acrProofLemmaMI}) holds also for $t < p$. \\
Defining  
\[
\mu = \max_{j \in \mathcal{M}} \| A_j + B_j K \|,
\]
we have   
\[
\|x(t)\| \leq \mu^t \|x(0)\|, \quad \forall t < p.
\]
If \( \mu \leq 1 \), we have  
\[
\|x(t)\| \leq \mu^t \|x(0)\| \leq \rho^t \rho^p \|x(0)\|, \quad \forall t < p.
\]
and in the case \( \mu > 1 \), we obtain  
\[
\|x(t)\| \leq \mu^t \|x(0)\| \leq \mu^{p-1} \rho^t \rho^p \|x(0)\|, \quad \forall t < p.
\]
Now, defining \( C' = \max(1, \mu^{p-1}) \), we can conclude that  
\[
\|x(t)\| \leq C' \rho^p \rho^t \|x(0)\|, \quad \forall t \in \mathbb{N}.
\]
Thus, (\ref{acrProofLemmaMI}) holds with  
\[
C = \frac{C'}{\rho^p}. 
\]

Sufficiency. Let us rewrite system~(\ref{eq:linSS}) in the form~(\ref{eq:NormalFormMI}). Following Theorem~\ref{thm:normalform}, one subsystem is fixed-time stable, while the other admits no non-trivial FTS subspaces; in this case, the minimal mode-independent growth rate of the latter subsystem equals zero. Now consider this subsystem. By Lemma~\ref{lemma:MIConj}, if its minimal mode-independent growth rate is zero, then there exist a nonzero state $x$ and a control input $u \in \mathbb{R}^m$ such that $A_j x + B_j u = 0$ for all $j \in \mathcal{M}$, 
that is, the system can be driven to the origin in one step. This would imply the existence of a nontrivial fixed-time stable subspace, contradicting the assumption that none exists. Therefore, the full system must be MIFTS.

\end{proof}
\section{Illustrative example}

In this section, we present a numerical example to illustrate the results discussed in Sections~\ref{sec:FTS} and~\ref{sec:ACR}.

\begin{myexample}
    Consider system~(\ref{eq:linSS}) with the following matrices \newline
$    A_1= \begin{bmatrix}
            -1 & -2& 2.5\\
 1& 1& -1\\
 0& 1& 0  \end{bmatrix}, A_2= \begin{bmatrix}
     -2 &-1 &3.5\\
 2& 2 &-1\\
 0& -1& 0.5\\
 \end{bmatrix}, B_1=\begin{bmatrix}
      -1\\
 0\\
 1
 \end{bmatrix}, B_2= \begin{bmatrix}
     1\\
 0\\
 -1
 \end{bmatrix}
   $.
We begin by showing that the system is MDFTS. Using the recursive construction in~(\ref{eq:algoMD}), we obtain the sequence of subspaces $E_k$ and find that $E_3 = \mathbb{R}^n$. By Theorem~\ref{thm:MDFTS}, this establishes that the system is MDFTS. The corresponding state-feedback control law takes the form $u(t) = K_{\sigma(t)} x(t)$, with mode-dependent gains derived from Lemma \ref{lemma:LemmaMD}: $K_1 = \begin{bmatrix}
    0 & -1 & -0.5
\end{bmatrix}, K_2 = \begin{bmatrix}
    0 & -1 & 0.5
\end{bmatrix}$. These feedback matrices ensure that any initial condition is driven to the origin within at most three steps, regardless of the switching sequence. The closed-loop evolution is depicted in Figure \ref{fig:MD}, where the trajectory is seen to successively enter the subspaces $E_2$ and $E_1$ before converging to the origin. 

\begin{figure}[H] 
    \centering
    \includegraphics[width=0.6\textwidth]{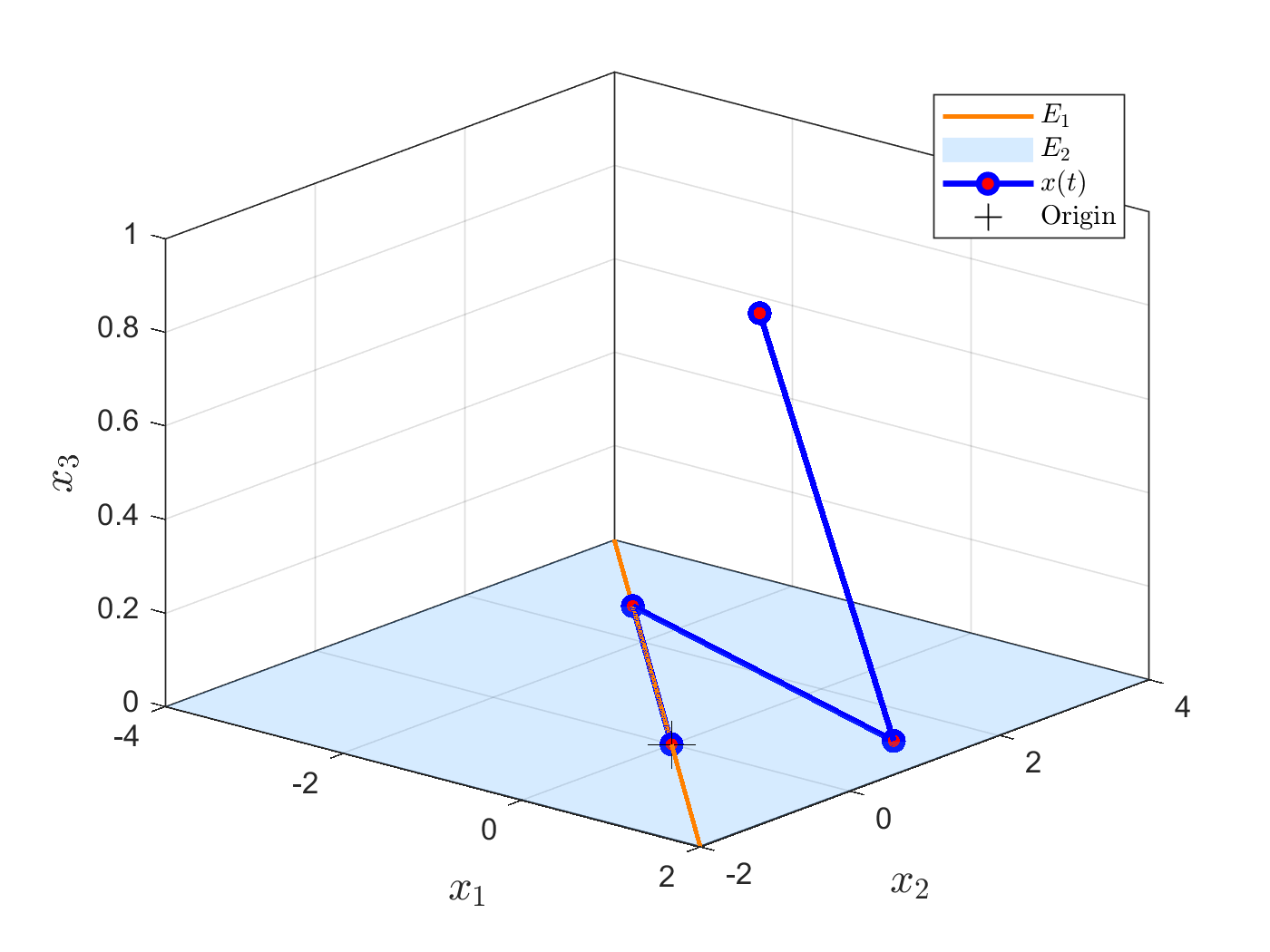}
    \caption{Closed-loop trajectory under mode-dependent controller. The trajectory reaches the origin in 3 time steps.}
    \label{fig:MD}
\end{figure}

However, the system is not MIFTS. In fact, from (\ref{eq:algoMI}) we can compute the largest set of initial states  from which the system can be fixed-time stabilized under mode-independent control and verify that $E_2$ is a subspace of $\mathbb{R}^3$ of dimension 2.

We now apply the normal form decomposition to this system.  
We perform a change of feedback of the form \( u(t) = Kx(t) + v(t) \), where \( K \) is computed according to Lemma~\ref{LemmaMIequivalences} and takes the value $K=\begin{bmatrix}
    0 & -1 & 0
\end{bmatrix}$. We also apply a change of coordinates \( z = Px \), where the matrix \( P \) is constructed by concatenating a basis of \( E_2\) with an arbitrary vector such that \( P \) is invertible. Let us choose
$
P = \begin{bmatrix}
1 & 2 & 1 \\
1 & 1 & 1 \\
0 & 0 & 1
\end{bmatrix}.
$
Under this transformation, the system can be written in the structured form (\ref{eq:NormalFormMI}) with matrices given by:

\[
A_1^{yy} = \begin{bmatrix}
0 & 1 \\
0 & 0
\end{bmatrix}, \quad
A_1^{y\xi} = \begin{bmatrix}
0 \\
2
\end{bmatrix}, \quad
A_1^{\xi\xi} = \begin{bmatrix}
0.5
\end{bmatrix}, \quad
B_1^y = \begin{bmatrix}
0 \\
0
\end{bmatrix}, \quad
B_1^\xi = \begin{bmatrix}
1
\end{bmatrix},
\]

\[
A_2^{yy} = \begin{bmatrix}
0 & 2 \\
0 & 0
\end{bmatrix}, \quad
A_2^{y\xi} = \begin{bmatrix}
0 \\
3
\end{bmatrix}, \quad
A_2^{\xi\xi} = \begin{bmatrix}
0.5
\end{bmatrix}, \quad
B_2^y = \begin{bmatrix}
0 \\
0
\end{bmatrix}, \quad
B_2^\xi = \begin{bmatrix}
-1
\end{bmatrix}.
\]

From Theorem~\ref{thm:normalform}, we conclude that the \( y \)-subsystem is fixed-time stable, which is also evident from the structure of the corresponding matrices. Moreover, the \( \xi \)-subsystem admits no non-trivial FTS subspace. In this example, the \( \xi \)-subsystem is one-dimensional, which allows for a straightforward computation of its minimal growth rate, yielding \( \rho_{\xi}^I = 0.5 \).
Therefore, by Theorem~\ref{thm:normalform}, we deduce that the minimal growth rate of the original system is also \( 0.5 \). Notice that the minimal growth rate is achieved for $v(t)=0$, as, without the knowledge of the active mode, the system’s evolution may exhibit growth, and the optimal control strategy in this setting is to apply no control. Due to this reason, and since $A_1^{\xi \xi}=A_2^{\xi \xi}=0.5$, the switching has no effect on the $ \xi $-subsystem, which behaves like a linear system with exponential decay rate 0.5.
The plot in Figure \ref{fig:MI} below shows the trajectories with \( v(t) = 0 \) for two different initial conditions. The trajectory starting from a point in \( E_2 \) converges in fixed time, specifically, in two steps, while, more generally, any initial condition in \( \mathbb{R}^3 \) leads to exponential convergence, as illustrated by the other trajectory in the figure.
Moreover, in the time-scale plot in Figure \ref{fig:MI_2} it can be seen, in fact, that the norm of the trajectory starting from a generic initial condition in \( \mathbb{R}^3 \) decays at rate 0.5.
\begin{figure}[H] 
    \centering
    \includegraphics[width=0.6\textwidth]{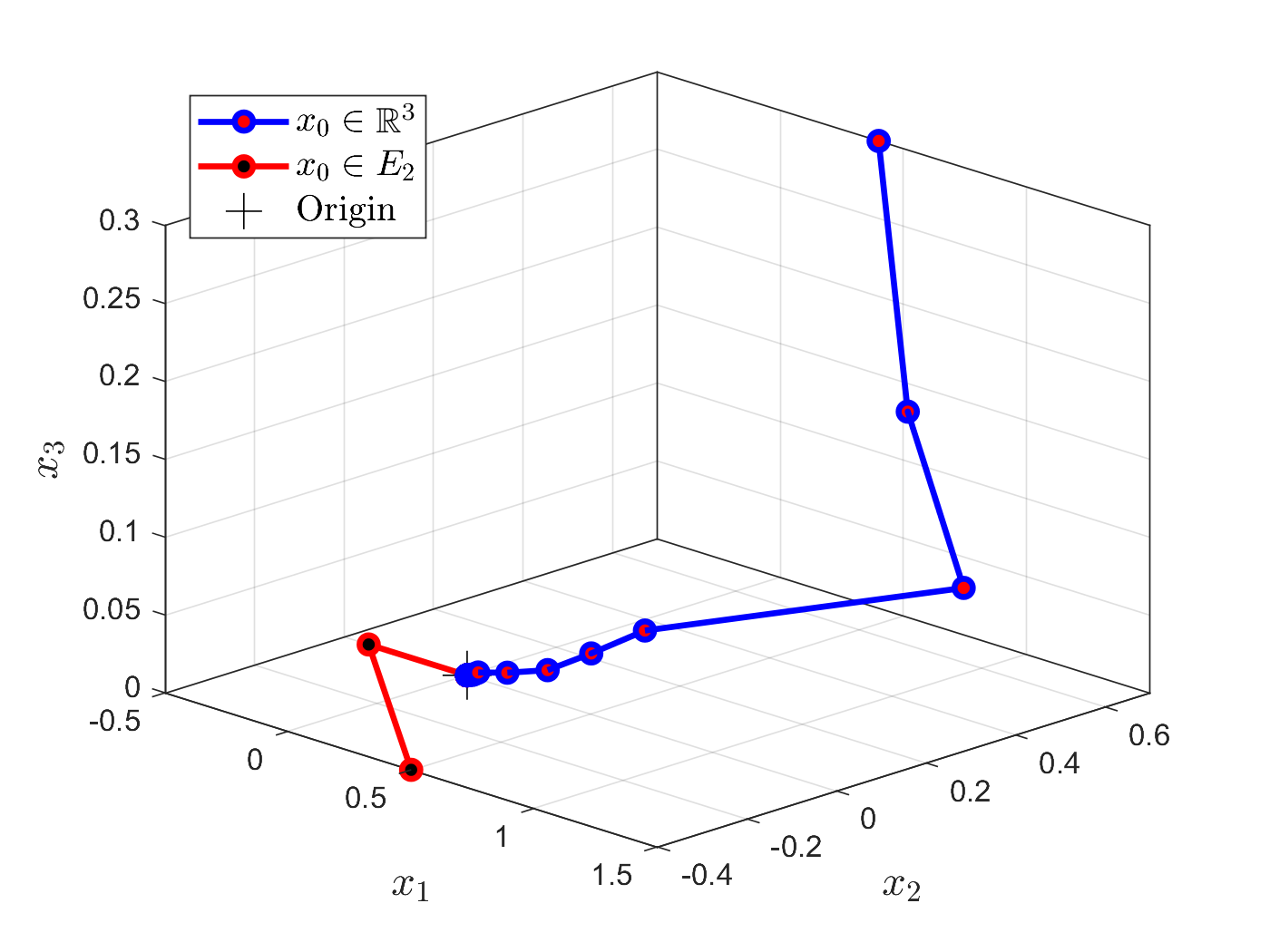}
    \caption{Closed-loop trajectories under mode-dependent controller. One trajectory exhibits exponential convergence, while the other reaches the origin in 2 time steps.}
    \label{fig:MI}
\end{figure}
\begin{figure}[H] 
    \centering
    \includegraphics[width=0.6\textwidth]{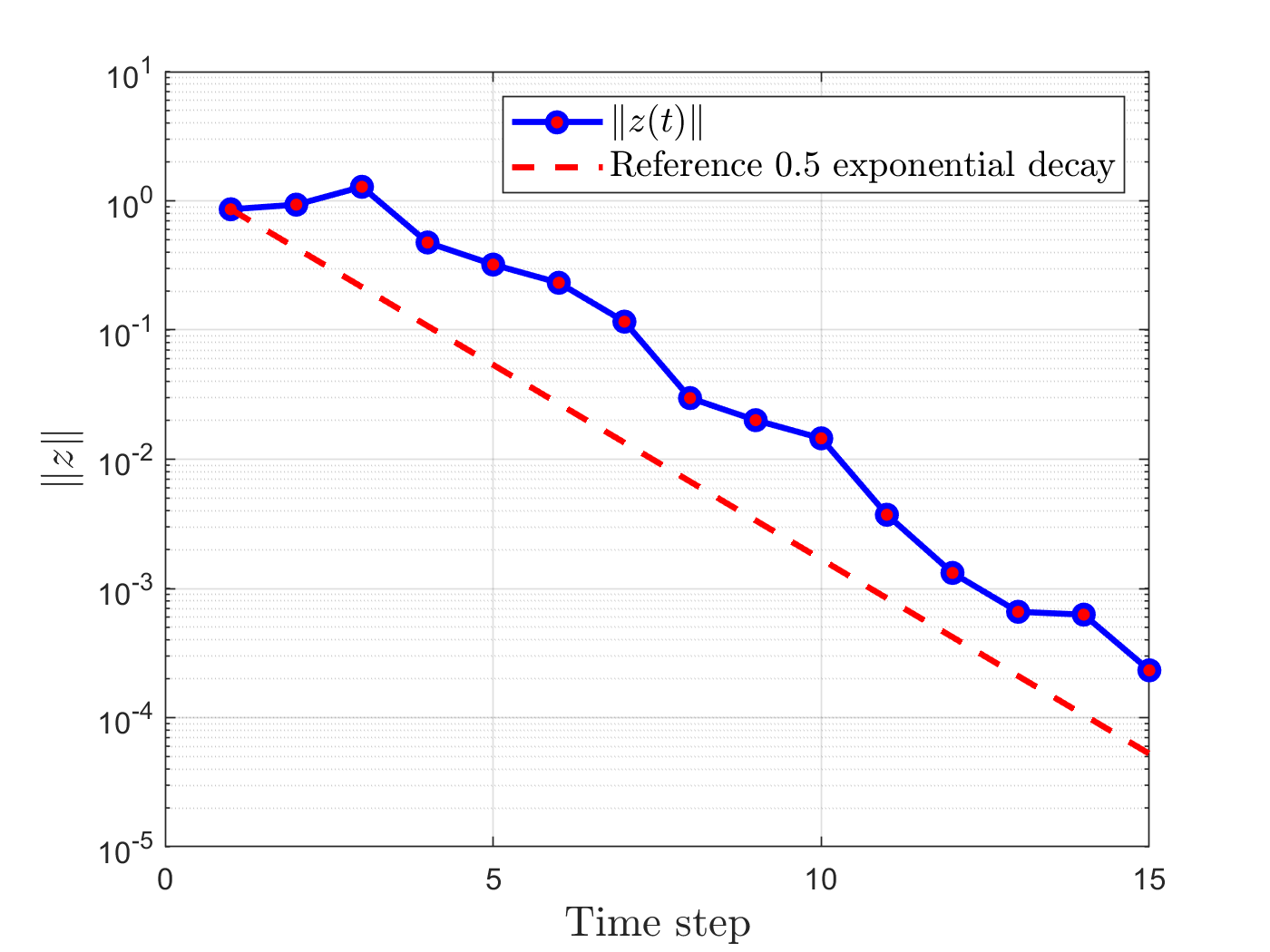}
    \caption{Time-scale plot of the closed-loop trajectory under mode-independent controller.}
    \label{fig:MI_2}
\end{figure}
\end{myexample}

\section{Conclusions}
In this work, we provided the characterization of discrete-time switched linear systems that admit arbitrarily fast exponential convergence. To this end, we first presented a systematic algorithm based on recursive subspace constructions, and we derived necessary and sufficient geometric conditions for both mode-dependent and mode-independent fixed-time stabilization. Then, we established the equivalence between fixed-time stabilizability and arbitrarily fast convergence, thereby extending a classical property from the autonomous case to systems with control inputs. As a further contribution, we introduced a normal form decomposition for switched linear systems, showing that any such system can be split into a fixed-time stable component and a residual subsystem that retains the original system's minimal growth rate.
Future work will address the constrained switching setting, with the goal of characterizing the largest classes of switching signals that permit arbitrarily fast stabilization, and identifying conditions under which subsets of switching sequences yield zero minimal convergence rate.

\bibliography{biblio} 
\bibliographystyle{plain}

\end{document}